\documentclass[a4paper,11pt]{amsart}
\usepackage{cases}
\usepackage{amsfonts,latexsym,rawfonts,amsmath,amssymb,amsthm,mathrsfs}
\usepackage[plainpages=false]{hyperref}
\usepackage{graphicx}

\RequirePackage{color}

\theoremstyle{plain}
  \newtheorem{theorem}{Theorem}[section]
  \newtheorem{lemma}{Lemma}[section]

\theoremstyle{remark} % this uses italics in the label and roman in the body

\theoremstyle{definition}

\numberwithin{equation}{section}

\renewcommand{\det}{\mbox{det}}
  
  \newcommand{\dist}{\mbox{dist}}

\begin{document}

\title[Noncompact $L_p$-Minkowski problems]{Noncompact $L_p$-Minkowski problems}

\author{Yong Huang}
\address{Institute of Mathematics, Hunan University, Changsha, 410082, CHINA}
\email{huangyong@hnu.edu.cn}

\author{Jiakun Liu}
\address
	{Institute for Mathematics and its Applications, School of Mathematics and Applied Statistics,
	University of Wollongong,
	Wollongong, NSW 2522, AUSTRALIA}
\email{jiakunl@uow.edu.au}

\thanks{Research of Huang was supported by NSF Grant 11001261, Research of Liu was supported by the Australian Research Council DP170100929}

%    General info

\date{\today}

%\dedicatory{ }


\begin{abstract}
In this paper we prove the existence of complete, noncompact convex hypersurfaces whose $p$-curvature function is prescribed on a domain in the unit sphere. 
This problem is related to the solvability of Monge-Amp\`ere type equations subject to certain boundary conditions depending on the value of $p$.
The special case of $p=1$ was previously studied by Pogorelov \cite{P2} and Chou-Wang \cite{CW95}.
Here, we give some sufficient conditions for the solvability for general $p\neq1$.
%Our results extend those of Chou and Wang on the special case of $p=1$. 
\end{abstract}

\maketitle

\baselineskip=16.4pt
\parskip=3pt

%%%%%%%%%%%%%%%% Sect 1 %%%%%%%%%%%%%%%%%%%%%
\section{Introduction}

Let $M$ be a compact, strictly convex $C^2$-hypersurface in $\mathbb{R}^{n+1}$. 
Since the Gauss map is a bijection between $M$ and the unit sphere $\mathbb{S}^n$, $M$ can be parametrised by the inverse of the Gauss map, and consequently the Gauss curvature $K$ of $M$ can be regarded as a function on $\mathbb{S}^n$.
Let $H$ be the support function of $M$ (see definitions in \S2).
For $p\in\mathbb{R}$, $K_p:=KH^{p-1}$ is called the $p$-curvature of $M$.
The $L_p$-Minkowski problem introduced by Lutwak \cite{L2} asks that whether a given function on $\mathbb{S}^n$ is the $p$-curvature of a unique compact convex hypersurface.
This problem is related to the solvability of the following Monge-Amp\`ere type equation
	\begin{equation}\label{eq1.1}
		\det\,(\nabla_{ij} H + H \delta_{ij}) = f H^{p-1} \qquad \mbox{on }\ \mathbb{S}^n,
	\end{equation}
where $\nabla$ is the covariant differentiation with respect to an orthonormal frame on $\mathbb{S}^n$.
When $p=1$, one has the classical Minkowski problem \cite{CY,Po}.
For general $p$, the $L_p$-Minkowski problem has been intensively studied in recent decades, for example, in  \cite{BLYZ,CW, HL, Lu18, LW, L2,L3,LYZt, Zh, Z, Zh1} and many others.
We refer the reader to the newly expanded book \cite{S} by Schneider for a comprehensive introduction on related topics.

The same problem makes perfectly sense for complete, noncompact, convex hypersurfaces. 
In that case, by a suitable rotation, the spherical image of such a hypersurface is an open convex subset contained in the hemisphere $\mathbb{S}^n_-:=\{X\in\mathbb{S}^n : X_{n+1}<0\}$.
The corresponding problem is then: Given an open convex subset $D$ of $\mathbb{S}^n_-$ and a positive function $K_p$ in $D$, does there exist a (unique) complete convex hypersurface with spherical image $D$ and $p$-curvature $K_p$?

When $p=1$, Pogorelov \cite{P2} firstly proved the existence of such a hypersurface under certain decay conditions on $K$ near the boundary of $D$.  Chou and Wang \cite{CW95} considered it in more general cases.
For $p\neq1$, this problem becomes much more complicated, partly because the $p$-curvature $K_p$ involves the support function $H$, which depends on the position of hypersurface $M$ and thus is not translation-invariant. 
In this paper we give sufficient conditions for the solvability for general $p\neq1$, and extend Chou-Wang's results in \cite{CW95} for $p=1$.

Similarly as above, the problem in noncompact setting is related to the solvability of Equation \eqref{eq1.1} in $D$ associated with certain compatible boundary conditions, where $f=K_p^{-1}$ is prescribed. 
One can see clearly from Equation \eqref{eq1.1} that whether $p>1$ or $p<1$ makes a big difference, as the right hand side of equation goes to degenerate or singular when $H\to 0$, respectively. 
Correspondingly, in the subsequent context we shall consider these two cases separately. 

\textbf{When $p<1$}, from a geometric observation we show that if $f\geq0$, there does not exist such a complete, noncompact, convex hypersurface (with $H\geq0$) satisfying Equation \eqref{eq1.1} (see Lemma \ref{nonexistence}).
Instead, we consider $H < 0$, namely the origin lies in the concave side of $M$, and the hypersurface $M$ satisfies 
	\begin{equation}\label{re p<1}
		H=-\hat fK^{\frac{1}{1-p}} \qquad (p<1)
	\end{equation}
for a given positive function $\hat f$ in $D$.
We remark that when $p<1$, it is necessary to have $D$ strictly contained in $\mathbb{S}^n_-$, see \S3.
Our first result is

\begin{theorem}\label{mt1.1}
Let $D$ be a uniformly convex $C^2$-domain strictly contained in $\mathbb{S}^n_-$, $\hat f$ a positive function in $C^\alpha(D)\cap L^{1-p}(D)$, where $\alpha\in(0,1)$ and $p<1$. 
Suppose there exists two positive functions $g$ and $h$ defined in $(0,r_0]$, $r_0>0$, satisfying
\begin{itemize}
\item[(a)] $\int_0^{r_0}\left( \int_s^{r_0} g^{1-p}(t) dt \right)^{1/n} ds < \infty$,
\item[(b)] $\int_0^{r_0} h^{1-p}(t)dt = \infty$,
\end{itemize}
so that $C^{-1} h(\dist(X,\partial D)) \leq \hat f(X)\leq C g(\dist(X,\partial D))$ near $\partial D$ for some constant $C>0$.
Then there exists a unique complete, noncompact, strictly convex hypersurface $M$ such that the support function $H\in C^{2,\alpha}(D)\cap C(\overline D)$ satisfies \eqref{re p<1} in $D$, and $H=0$ on $\partial D$.
\end{theorem}

We remark that the assumption $\hat f \in L^{1-p}(D)$ ensures $H=0$ on $\partial D$ and $M$ approaches to an asymptotic convex cone. Without this assumption, one can also obtain the existence of $M$ with a general boundary condition $H=\Phi$ for a function $\Phi\in C^2(\partial D)$, but to show $M$ is complete, one needs a stronger assumption that $h(\dist(X,\partial D)) / \hat f(X) \to 0$ as $X \to \partial D$. More details are contained in \S3.

\textbf{When $p>1$}, depending on the relative position of $D$ there are multiple cases for discussion.
By suitably rotating axes, we may assume $D$ satisfies one and exactly one of the following conditions:
\begin{itemize}
	\item[(I)]	$D$ is strictly contained in $\mathbb{S}^n_-$,
	\item[(II)]	$D=\mathbb{S}^n_-$,
	\item[(III)]	$D$ is a proper subset of $\mathbb{S}^n_-$ and it is not strictly contained in any hemisphere. 
\end{itemize}
We shall say $M$ is of \emph{type I, II, or III} when its spherical image $D$ satisfies (I), (II), or (III), respectively. 
Notice that by our choice of coordinates, $M$ is the graph of a convex function over a convex domain in the $(x_1,\cdots,x_n)$-space. 

For type \emph{I} hypersurfaces, it is clear that $M$ is complete if and only if $M$ is a graph over $\mathbb{R}^n$.
Correspondingly, we impose a boundary condition to the support function $H=\Phi$ on $\partial D$, where $\Phi$ is a prescribed function. 

\begin{theorem}\label{thm I}
Let $D$ be a uniformly convex $C^2$-domain satisfying condition (I), $p\geq1$ and $p \neq n+1$. 
Assume $K_p$ is a positive function in $C^\alpha(D)$ and $\Phi$ is a function in $C^{2}(\partial D)$.
Suppose there exists two positive functions $g$ and $h$ defined in $(0,r_0]$, $r_0>0$, satisfying
\begin{itemize}
\item[(a)] $\int_0^{r_0}\left(\int_s^{r_0}g(t)dt\right)^{1/n}ds<\infty$,
\item[(b)] $\int_0^{r_0}h(t)dt=\infty$, 
\end{itemize}
so that $K_p(X)g(\dist(X,\partial D)) \geq C^{-1}$ and $K_p(X)h(\dist(X,\partial D)) \leq C$ near $\partial D$ for some constant $C>0$. 
Then there exists a unique complete, noncompact, strictly convex hypersurface $M$ such that $K_p$ is the $p$-curvature of $M$ and $H=\Phi$ on $\partial D$. 
\end{theorem}

In fact, for this reconstruction of complete convex hypersurface, Aleksandrov \cite{Al42} firstly formulated the geometric problem with prescribed area of Gaussian mapping and its asymptotic cone.
It amounts to the solvability of the boundary value problem for a Monge-Amp\`ere equation, see \cite{Al42, AlB}. 
Bakelman \cite{Ba58, Ba} established the existence and uniqueness of convex generalised solutions for the second boundary value problem of the Monge-Amp\`ere equation
\begin{equation}
\det\,D^2u =\frac{T(x)}{Q(Du)}.
\end{equation}
It has been convinced that the necessary and sufficient condition for the existence of those complete hypersurfaces project one-to-one on $\mathbb{R}^n$ with prescribed asymptotic cone $\mathcal{C}$, is
$$
\int_{\mathbb{R}^n} T(x)dx=\int_{\aleph_{\mathcal{C}}(\mathbb{R}^n)}Q(p)dp,
$$
where $\aleph_{\mathcal{C}}(\mathbb{R}^n)$ is the normal image of $\mathcal{C}$. See also Pogorelov \cite{PB, P2}. 
Moreover, Oliker \cite{O} constructed a minimisation problem associated with the Monge-Kantorovich optimal mass transfer problem for this kind of reconstruction geometric problem.

For type \emph{II} hypersurfaces, we prove the existence of such a complete hypersurface $M$ when $p>n+1$, under an asymptotic growth assumption on the prescribed function $K_p$.
\begin{theorem}\label{thm II}
Assume that $K_p$ satisfies an asymptotic growth condition
	\begin{equation}\label{Kp asym}
		K_p(X) \sim |X_{n+1}|^{2q},
	\end{equation}
for some constant $q\in(0,1)$. 
When $p>n+1$, there exists a complete noncompact convex hypersurface $M$ such that $K_p$ is the $p$-curvature of $M$. 
\end{theorem}
We further remark that when $M$ is a graph over a bounded domain $\Omega^*$, \eqref{Kp asym} is necessary for the solvability, see \S4.2.

For type \emph{III} hypersurfaces, we have a similar result as Theorem \ref{thm I}, however, the boundary condition is imposed on part of $\partial D$ that is away from $\partial\mathbb{S}^n_-$. The corresponding statement is postponed to Theorem \ref{thm III} in \S4.3.

Last, we point out that the $L_p$-Minkowski problem is related to the expanding Gauss curvature flow when $p>1$, and the contracting Gauss curvature flow when $p<1$, as Equation \eqref{eq1.1} describes homothetic solutions in each case, respectively. 
For complete, noncompact hypersurfaces, one may consult Urbas \cite{U1,U2} for works in this direction, also \cite{CW95, HLX} and references therein. 

This paper is organised as follows: 
In \S2, we derive the Monge-Amp\`ere equation \eqref{eq1.1}. Although this has been done in some literatures, we would like to provide a more general and unified treatment, which includes \eqref{eq1.1} as a special case when the metric on $\mathbb{S}^n$ is orthonormal. 
In \S3, we consider the case of $p<1$ and prove Theorem \ref{mt1.1}. Moreover, when the prescribed function $\hat f$ satisfies certain boundedness conditions, we can give a different proof without the uniform convexity assumption on $D$. 
In \S4, we consider the case of $p>1$, Theorems \ref{thm I}, \ref{thm II} and \ref{thm III} are proved in \S4.1--\S4.3, respectively.

 %%%%%%%%%%%%%%%%%%%% Sect 2 %%%%%%%%%%%%%%%%%%%%%%%%%%%%%%%%%%%%%%%
\vspace{10pt}

\section{Preliminaries}

%%%%%%%%%%%%%%%%%%%% Sect 2.1 %%%%%%%%%%%%%%%%%%%%%%%
\subsection{Support function}
Let $M$ be a strictly convex $C^2$-hypersurface in $\mathbb{R}^{n+1}$ and let $D\subset\mathbb{S}^n$ be its spherical image.
Assume that $M$ is parametrised by the inverse Gauss map $X : D \to M \subset \mathbb{R}^{n+1}$. 
The \emph{support function} of $M$ is defined by
	\begin{equation}\label{supp}
		H(\xi) = \langle \xi, X(\xi) \rangle,\quad \xi\in D,
	\end{equation}
where $\langle\cdot,\cdot\rangle$ is the inner product in $\mathbb{R}^{n+1}$. 

The metric and the second fundamental form of $M$ can be represented in terms of its support function $H$. 
To see that, let $(u_1,\cdots,u_n)$ be a local coordinate chart and $\{e_i=\partial_i\xi\}$ be the local frame field on $D$, where $\partial_i=\partial/\partial u_i$, $i=1,\cdots,n$.  
By differentiating \eqref{supp} we obtain that
	\begin{equation}\label{dif1}
		\partial_i H = \langle \partial_i\xi, X \rangle + \langle \xi, \partial_iX \rangle = \langle e_i, X \rangle,
	\end{equation}
since $\partial_i X(\xi)$ is tangential to $M$ at $X(\xi)$, and $\xi$ is the normal to $M$ at $X(\xi)$.
Differentiating once again, we have
	\begin{equation}\label{dif2}
		\partial_{ij} H = \langle \partial_je_i,X \rangle + \langle e_i, \partial_j X\rangle = \langle \partial_je_i,X \rangle + h_{ij},
	\end{equation}
where $h_{ij}$ is the second fundamental form of $M$. 

To compute $\langle \partial_je_i,X \rangle$, we use the Gauss derivation formulas to get
	\begin{equation}\label{GaEq}
		\partial_je_i = \Gamma_{ij}^k e_k - \sigma_{ij}\xi,
	\end{equation}
where $\Gamma_{ij}^k$ are the Christoffel symbols, and $\sigma_{ij}=\langle e_i, e_j\rangle$ is the metric on $D\subset\mathbb{S}^n$, respectively. 
Combining \eqref{dif1}, \eqref{dif2} and \eqref{GaEq}, we can obtain
	\begin{equation*}
	\begin{split}
		\partial_{ij} H &= \Gamma_{ij}^k\langle e_k, X \rangle - \sigma_{ij}\langle \xi, X \rangle + h_{ij} \\
			&= \Gamma_{ij}^k \partial_k H - H \sigma_{ij} + h_{ij},
	\end{split}
	\end{equation*}
and thus
	\begin{equation}\label{2ndf}
		h_{ij} = \partial_{ij} H - \Gamma_{ij}^k \partial_k H + H\sigma_{ij} =  \nabla_{ij} H + H\sigma_{ij}, 
	\end{equation}
where $\nabla_{ij} H $ is the covariant differentiation of $H$ with respect to the frame $\{e_i\}$ on $\mathbb{S}^n$. 

Let's now compute the metric $g_{ij}$ of $M$. Using the Gauss-Weingarten relations
	\begin{equation}\label{GWEq}
		\partial_i\xi = h_{ik}g^{kl}\partial_lX,
	\end{equation}
we have
	\begin{equation*}
		\sigma_{ij} = \langle \partial_i\xi, \partial_j\xi \rangle = h_{ik}g^{kl}h_{jm}g^{ms}\langle \partial_lX, \partial_sX \rangle = h_{ik}h_{jl}g^{kl}.
	\end{equation*}
Hence, the metric satisfies
	\begin{equation}\label{1stf}
		g_{ij} = h_{ik}h_{jl}\sigma^{kl}.
	\end{equation}

The principal radii of curvature are the eigenvalues of the matrix $b_{ij}=h^{ik}g_{jk}$, which, by virtue of \eqref{2ndf} and \eqref{1stf}, is given by
	\begin{equation}\label{eign}
		b_{ij} = h_{ik}\sigma^{kj} = (\nabla_{ik} H + H \sigma_{ik})\sigma^{kj}.
	\end{equation}
Therefore, the Gauss curvature $K$ of $M$ is given by
	\begin{equation}\label{Gaus}
		\frac{1}{K} = \det\,b_{ij} = \frac{\det\,(\nabla_{ij} H + H \sigma_{ij})}{\det\,(\sigma_{ij})}.
	\end{equation}

For a general $p\in\mathbb{R}$, the \emph{$p$-curvature} of $M$ is defined by $K_p:=KH^{p-1}$.
When $p=1$, $K_1=K$ is the Gauss curvature. When $p\neq1$, $K_p$ involves the support function $H$, and thus is not intrinsic. 
The $L_p$-Minkowski problem asks for the existence of $M$ with a prescribed $p$-curvature $K_p$. 
Let $f:=K_p^{-1}=H^{1-p}/K$, from \eqref{Gaus} one can obtain the Monge-Amp\`ere equation satisfied by $H$,
	\begin{equation}\label{PDEH}
		\det\,(\nabla_{ij} H + H \sigma_{ij}) = f H^{p-1} \det\,(\sigma_{ij})\quad\mbox{on }D.
	\end{equation}

In particular, under a smooth local orthonormal frame field on $\mathbb{S}^n$, namely $\sigma_{ij}=\delta_{ij}$, the above equation becomes \eqref{eq1.1}, namely
	\begin{equation}\label{pde2}
		\det\,(\nabla_{ij} H + H \delta_{ij}) = f H^{p-1}\quad\mbox{on }D.
	\end{equation}

%%%%%%%%%%%%%%%%%%%%%% Sect 2.2 %%%%%%%%%%%%%%%%%%%%%%%%
\subsection{Homogeneous extension}
In order to study the solvability of $L_p$-Minkowski problems, it is convenient to express the above equations for $H$ in a local coordinate chart. 
Extend $H$ to be a function of homogeneous degree one over the cone $\{\lambda\xi : \xi\in D, \lambda>0\}$, let $\Omega := \{\lambda\xi : \xi\in D, \lambda>0\} \cap \{x_{n+1}=-1\}$, and $u(x) = H(x,-1)$, where $x=(x_1,\cdots,x_n)$. Denote
	\begin{equation}\label{mufa}
		\mu(x) = \left(1+\sum_{i=1}^nx_i^2\right)^{1/2}.
	\end{equation}
By the homogeneity, 
	\begin{equation}\label{Htou}
		u(x) = \mu(x) H(\xi(x)),\quad x\in\Omega,
	\end{equation}
where $\xi(x)\in D$ is given by
	\begin{equation}\label{qtop}
		\mu(x)\xi(x) = (x,-1).
	\end{equation}
	
In order to rewrite equation \eqref{PDEH} in terms of $u$, we adopt the following computations from \cite{LO}. 
Differentiating \eqref{qtop} we have
	\begin{equation}\label{dmu1}
		(\partial_i\mu)\xi + \mu \, \partial_i\xi = (\underbrace{0,\cdots,0}_{i-1},1,0,\cdots,0).
	\end{equation}
Differentiating once again yields
	\begin{equation}\label{dmu2}
		(\partial_{ij}\mu)\xi + (\partial_i\mu)\partial_j\xi + (\partial_j\mu)\partial_i\xi + \mu \, \partial_{ij}\xi = 0.
	\end{equation}
By the Gauss derivation formulas,
	\begin{equation}
		\partial_{ij}\xi = \Gamma_{ij}^k\partial_k\xi - \sigma_{ij}\xi,
	\end{equation}
where $\sigma_{ij}=\langle\partial_i\xi,\partial_j\xi\rangle$ is the metric of $\mathbb{S}^n$.
Taking the inner product of \eqref{dmu2} with $\partial_s\xi$, and noting that $\langle\partial_s\xi,\xi\rangle=0$, we get
	\begin{equation}\label{aux1}
		\partial_i\mu \, \sigma_{sj}+ \partial_j\mu \, \sigma_{si} + \mu \, \Gamma_{ij}^k \, \sigma_{ks} = 0.
	\end{equation}
While taking the inner product of \eqref{dmu2} with $\xi$, we also get
	\begin{equation}\label{aux2}
		\partial_{ij}\mu = \mu \, \sigma_{ij}.
	\end{equation}
	
Now, differentiating \eqref{Htou} we obtain
	\begin{equation*}
		\partial_i u = \partial_i\mu \, H + \mu \, \partial_iH,
	\end{equation*}
then by \eqref{aux1} and \eqref{aux2}
	\begin{equation}\label{2ndu}
	\begin{split}
		\partial_{ij} u &= \partial_{ij}\mu \, H + \partial_i\mu \, \partial_jH + \partial_j\mu \, \partial_iH + \mu \, \partial_{ij}H \\
			&= \partial_{ij}\mu \, H + (\partial_i\mu\,\sigma_{js} + \partial_j\mu\,\sigma_{is})\sigma^{sl}\partial_lH + \mu \, \partial_{ij}H \\
			&= \partial_{ij}\mu \, H - \mu\, \Gamma_{ij}^l \, \partial_lH + \mu \, \partial_{ij}H \\
			&= \partial_{ij}\mu \, H + \mu\, \nabla_{ij}H \\
			&= \mu \, (\nabla_{ij}H + H \sigma_{ij}).
	\end{split}
	\end{equation}
	
On the other hand, by straightforward computations we have
	\begin{equation}\label{cha1}
		\sigma_{ij} = (1+|x|^2)^{-1}\left(\delta_{ij}-\frac{x_ix_j}{1+|x|^2}\right),
	\end{equation}	
and
	\begin{equation}\label{cha2}
		\det\,(\sigma_{ij}) = (1+|x|^2)^{-(n+1)} = \mu^{-2(n+1)}.
	\end{equation}

Substituting \eqref{Htou}, \eqref{2ndu} and \eqref{cha2} into \eqref{PDEH}, we obtain the standard Monge-Amp\`ere equation satisfied by $u$,
	\begin{equation}\label{PDEu}
		\det\,D^2u = (1+|x|^2)^{-\frac{n+p+1}{2}} f\left(\frac{x,-1}{\sqrt{1+|x|^2}}\right)\,u^{p-1},\quad x\in\Omega,
	\end{equation}
where $D^2u=(\partial_{ij}u)$ is the Hessian matrix of $u$.
Therefore, the solvability of $L_p$-Minkowski problems is equivalent to the solvability of the Monge-Amp\`ere equation \eqref{PDEu}. 

In the complete, noncompact case with certain boundary conditions, whenever a convex solution of \eqref{PDEu} is given, as a rescaled support function it determines the hypersurface $M$ in the following way (see \cite{CY,Po}): 
Let $\Omega^*=Du(\Omega)$ and
	\begin{equation}\label{dual graph}
		u^*(y) = \sup\{ \langle x,y \rangle - u(x) : x\in\Omega \},\quad y\in\Omega^*.
	\end{equation}
Then $M$ is the graph $\{(y,u^*(y)) : y\in\Omega^*\}$, and its $p$-curvature is equal to $f^{-1}=K_p$ as prescribed.

By straightforward computations, the dual function $u^*$ satisfies
	\begin{equation}\label{PDEu*}
		\det\,D^2u^* = \frac{(1+|Du^*|^2)^{\frac{n+p+1}{2}}}{(y\cdot Du^*-u^*)^{p-1}}\, f^{-1}(\gamma),\quad y\in\Omega^*,
	\end{equation}
where $\gamma=\frac{(Du^*,-1)}{\sqrt{1+|Du^*|^2}}$ is the unit normal of $M$ at the point $(y,u^*(y))$.

%%%%%%%%%%%%%%%%%%%%% Sect 3 %%%%%%%%%%%%%%%%%%%%%%%%%%%%%%%%%%%%%%%%%%%
\vspace{10pt}

\section{The case of $p<1$}

In this section we first show a nonexistence result for hypersurface $M$ with support function $H\geq0$. 
Then alternatively, we consider the hypersurface satisfying \eqref{re p<1} with $H<0$ and prove Theorem \ref{mt1.1}.
When the prescribed function $\hat f$ satisfies further boundedness conditions, we also give some independent and interesting results for the existence and completeness. 
Throughout this section we assume $p<1$. 

Originally, one asks for a strictly convex $C^2$-hypersurface $M$ in $\mathbb{R}^{n+1}$ such that
	\begin{equation}\label{p<1}
		H = (f K)^{\frac{1}{1-p}},
	\end{equation}
for some prescribed function $f\geq0$ on the spherical image $D\subset\mathbb{S}^n_-$, where $H$ is the support function of $M$. 

\begin{lemma}[Nonexistence] \label{nonexistence}
If $f\in L^{1}(D)$ is a nonnegative function, there does not exist any complete, noncompact, strictly convex hypersurface $M$ satisfying \eqref{p<1}. 
\end{lemma}

\begin{proof}
Since $M$ is convex, $K\geq0$. From assumption $f\geq0$, the support function $H\geq0$ by \eqref{p<1}.
If $H=0$ on $\partial D$, then either $M$ is a cone or $M$ is degenerate with zero $n$-dimensional Hausdorff measure, $\mathcal{H}^n(M)=0$. 
Therefore, we assume that $\sup_{\partial D}H\geq\delta$ for some positive constant $\delta$. 
By continuity of $H$, there exists a subset $E\subset D$ such that $H\geq\delta/4$ in $E$.
Let $G:=X(E)\subset M$, where $X$ is the inverse Gauss map, see \eqref{supp}. 
Since $M\in C^2$ is strictly convex, the map $X : D\to M$ is a bijection.
Since $M$ is complete noncompact and $E\cap\partial D\neq\emptyset$, we have $\mathcal{H}^n(G)=\infty$, \cite{Wu}. Then by integration we obtain
	\begin{equation}\label{contra}
	\begin{split}
		\infty &= \int_G (\delta/4)^{1-p} \, d\mathcal{H}^n \leq \int_E \frac{H^{1-p}}{K} dx \\
		&\leq \int_D \frac{H^{1-p}}{K} dx = \int_D f dx <\infty,
	\end{split}
	\end{equation}
where $dx$ is the spherical measure of $\mathbb{S}^n$.
The last equality is due to \eqref{p<1}.
This is a contradiction to the assumption $f\in L^1(D)$, and thus Lemma \ref{nonexistence} is proved.
\end{proof}

We remark that in proving the above lemma, one can in fact show that the set $\{\xi\in D : H(\xi)=0\}$ has zero $\mathcal{H}^n$ measure, where $H\geq0$ is the support function of a complete, noncompact, strictly convex hypersurface $M$. 
Hence, the contradiction \eqref{contra} will occur under the assumption of Lemma \ref{nonexistence}.

Therefore, in the case of $p<1$, it is reasonable to consider the existence of hypersurface $M$ satisfying \eqref{re p<1} and $H\leq0$, that is
	\begin{equation}\label{p<1a}
		H = -\hat f K^{\frac{1}{1-p}}.
	\end{equation}
where $\hat f\geq0$ is a prescribed function on the spherical image $D$. 
Then one's aim is to look for a complete, noncompact, strictly convex hypersurface $M$ satisfying \eqref{p<1a}.
And in such cases, $K_p={\hat{f}}^{p-1}=K(-H)^{p-1}$ is the $p$-curvature of $M$. 

Note that $H\leq0$ implies that any tangent hyperplane $T$ to $M$ either contains the origin, or else, the origin lies on the opposite side of $T$ from $M$. 
If $0\in M$, it follows that every tangent hyperplane to $M$ must contains $0$ and $H\equiv0$, which implies that $M$ is a hyperplane containing the origin or a cone with vertex at the origin, but this contradicts with the strict convexity of $M$. 
Now assume that $0\notin M$, the origin lies on the opposite sides of all tangent hyperplanes from $M$, or equivalently, $H<0$ in $D$.  

Let $\mathcal{C}$ be the intersection of all closed halfspaces $\mathcal{P}$ of $\mathbb{R}^{n+1}$ with $0\in\partial\mathcal{P}$ and $M\subset\mathcal{P}$. 
Then $\mathcal{C}$ is a closed convex cone with vertex at the origin with nonempty interior. 
Moreover, $\partial\mathcal{C}$ can be represented as the graph of a convex degree one homogeneous function $\psi$ with $\psi\geq0$ in $\mathbb{R}^n-\{0\}$ and $|D\psi|$ bounded.
Recall that $M$ is a graph of $u^*$ over $\Omega^*\subset\mathbb{R}^n$, from the construction of $\mathcal{C}$, $Du^*(\Omega^*)\subset D\psi(\mathbb{R}^n)$.
Because $M$ is complete and $Du^*$ is bounded, we must have $\Omega^*=\mathbb{R}^n$, namely $M$ is an entire graph over $\mathbb{R}^n$.
By parallel translating supporting hyperplanes between $M$ and $\mathcal{C}$, one also has
	\begin{equation*}
		\overline{Du^*(\mathbb{R}^n)} = D\psi(\mathbb{R}^n),
	\end{equation*}
see \cite{U1} for more geometric details. 
Therefore, the spherical image of $M$,
	\begin{equation*}
		D = \frac{(Du^*,-1)}{\sqrt{1+|Du^*|^2}}(\mathbb{R}^n)
	\end{equation*}
must be strictly contained in $\mathbb{S}^n_-$, and its projection image $\Omega=\{\lambda\xi : \xi\in D, \lambda>0\}\cap\{x_{n+1}=-1\}$ must be a bounded, convex domain in $\mathbb{R}^n$.

Next lemma shows that under hypotheses of Theorem \ref{mt1.1}, $M$ is asymptotically approaching to $\mathcal{C}$ in the sense that $H(X)\to0$ as $X\to\partial D$.
\begin{lemma}[Asymptotic] \label{Asymptotic}
Under the hypotheses of Theorem \ref{mt1.1}, let $M$ be a solution satisfying \eqref{p<1a}. 
If $\hat f \in L^{1-p}(D)$, then $H(X)\to0$ as $X\to\partial D$.
\end{lemma}

\begin{proof}
Suppose if not true, by continuity of $H$, there exists some $X_0\in\partial D$ and a positive constant $\delta$ such that $-H\geq\delta$ in a neighborhood of $X_0$, $E\subset D$. 
Let $G:=X(E)\subset M$, where $X$ is the inverse Gauss map, which is a bijection from $D$ to $M$.
As $X_0\in E\cap\partial D$, $E\cap\partial D\neq\emptyset$. 
Since $M$ is complete noncompact, one has $\mathcal{H}^n(G)=\infty$, \cite{Wu}.
Then similarly to Lemma \ref{nonexistence}, by integration we have
	{\begin{equation}\label{contra2}
	\begin{split}
		\infty &= \int_G (\delta/4)^{1-p} \, d\mathcal{H}^n \leq \int_E \frac{(-H)^{1-p}}{K} dx \\
		&\leq \int_D \frac{(-H)^{1-p}}{K} dx= \int_D \hat f^{1-p} dx <\infty,
	\end{split}
	\end{equation}
where $dx$ is the spherical measure of $\mathbb{S}^n$.
The last equality is due to \eqref{p<1a}.}
By assumption $\hat f\in L^{1-p}(D)$, this contradiction thus implies that $H(X)\to0$ as $X\to\partial D$.
\end{proof}

In a local coordinate chart, by \eqref{Htou},
	\begin{equation*}
		u = \mu H < 0 \quad \mbox{in }\Omega.
	\end{equation*}
Since $M$ is enclosed by the asymptotic cone $\mathcal{C}$, $u=0$ on $\partial\Omega$.
It is also clear that $M$ is complete if and only if $Du(\Omega)=\mathbb{R}^n$.

From the computations in \S2, the above question \eqref{p<1a}  is related to a variant of Monge-Amp\`ere equation  
	\begin{equation*}
		\mu^{-1}u = -\hat f \mu^{-\frac{n+2}{1-p}}\left(\det\,D^2u\right)^{-\frac{1}{1-p}}\quad \mbox{in }\Omega.
	\end{equation*}
Hence, by Lemma \ref{Asymptotic} we pose the following boundary value problem 
	\begin{eqnarray}
		\det\,D^2u \!\!&=&\!\! \mu^{-(n+p+1)} \hat f^{1-p} \left(\frac{-1}{u}\right)^{1-p}\quad \mbox{in }\Omega, \label{PDE p<1}\\
		u \!\!&=&\!\! 0 \qquad  \mbox{on }\partial\Omega, \label{bc0 p<1} \\
		|Du(x)| \!\!&\to&\!\! \infty \qquad \mbox{as } x \to \partial\Omega. \label{bc com}
	\end{eqnarray}
Here, $\hat f(x)$ is interpreted as $\hat f\left(\frac{x,-1}{\sqrt{1+|x|^2}}\right)$ for $x\in\Omega$.

Therefore, in order to prove Theorem \ref{mt1.1}, it suffices to prove the following result.
\begin{theorem} \label{thm p<1}
Let $\Omega$ be a bounded, uniformly convex smooth domain in $\mathbb{R}^n$, and $\hat f\geq0$ be a smooth function in $\Omega$.
Suppose there exists two positive functions $g$ and $h$ satisfying conditions (a) and (b) in Theorem \ref{mt1.1} such that $C^{-1}h(\dist(x,\partial\Omega)) \leq \hat f(x)\leq Cg(\dist(x,\partial\Omega))$ near $\partial\Omega$ for some constant $C>0$.
Then \eqref{PDE p<1}--\eqref{bc com} has a unique smooth solution $u$.\end{theorem}

The Dirichlet problem \eqref{PDE p<1}--\eqref{bc0 p<1} was previously studied by Cheng-Yau \cite{CY}. They proved that if $\Omega$ satisfies a uniform enclosing sphere condition and $\hat f(x)\leq C\dist(x,\partial\Omega)^{\beta-n-1}$, $\beta>0$, then there admits a unique solution. If $\hat f\equiv 1$, while $\Omega$ is merely a bounded convex domain, Urbas \cite{U1} also obtained the unique existence of convex solution. One can easily check that the assumption on $\hat f$ in \cite{CY} is contained in condition (a) of Theorem \ref{thm p<1}.
We divide the proof of Theorem \ref{thm p<1} into two parts: Prove the solvability of Dirichlet problem \eqref{PDE p<1}--\eqref{bc0 p<1}, and verify the solution satisfies boundary condition \eqref{bc com}.

%%%%%%%%%%%%%%%% Sect 3.1 %%%%%%%%%%%%%%%%%
\subsection{Existence}

For the existence part, we consider a general Dirichlet boundary condition
	\begin{equation}\label{bc1 p<1}
		u = \phi \quad\mbox{ on }\partial\Omega,
	\end{equation}
where $\phi\leq0$ is a convex function in $\Omega$. 
Write
	\begin{equation}\label{def r}
		R(x) = \mu^{-(n+p+1)} \hat f^{1-p}(x), \quad x\in\Omega.
	\end{equation}
Equation \eqref{PDE p<1} can be simplified to
	\begin{equation} \label{PDE p<1 s}
		\det\,D^2u = R(x) \left(\frac{-1}{u}\right)^{1-p}.
	\end{equation}
	
Let $\Omega_r=\{x\in\Omega : \dist(x,\partial\Omega)>r\}$. When $\Omega$ is uniformly convex, for $r_0>0$ small depending on the geometry of $\Omega$, $\Omega_r$ is still uniformly convex. For $x\in\Omega\setminus\Omega_{r_0}$, $x$ can be represented uniquely by $x_b+d n(x_b)$, where $x_b\in\partial\Omega$, $d=\dist(x,\partial\Omega)$, and $n(x_b)$ is the unit inner normal at $x_b$. For a function $f$ defined near $\partial\Omega$ we write $f(x)=f(x_b,d)$.
	
\begin{lemma} \label{lexi}
Let $\Omega$ be a bounded, uniformly convex $C^2$-domain. 
Suppose there exists a positive function $g$ in $(0,r_0]$ satisfying
	\begin{equation}\label{cond g}
		\int_0^{r_0}\left( \int_s^{r_0} g(t) dt \right)^{1/n} ds < \infty
	\end{equation}
such that 
	\begin{equation*}
		R(x_b,d) \leq g(d).
	\end{equation*}
Then \eqref{PDE p<1 s} admits a unique generalised solution $u$ satisfying \eqref{bc1 p<1}.
\end{lemma}

\begin{proof}
Our proof is inspired by \cite{CW95}.
For $x=x_b+dn(x_b)$ in $\Omega\setminus\Omega_{r_0}$, we define
	\begin{equation}
		v(x) = \tilde\rho(d) := -(-\rho(d))^\varepsilon,
	\end{equation}
where $\varepsilon\in(0,1)$ is a constant to be determined and 
	\begin{equation*}
		\rho(d) := - \int_0^d\left( \int_s^{r_0} g(t) dt \right)^{1/n} ds.
	\end{equation*} 

By computations, see \cite[Lemma 1]{CW95},
	\begin{equation*}
		\det\,D^2v(x) =\prod_{i=1}^{n-1}\frac{k_i(x_b)}{1-k_i(x_b)d}(-\tilde\rho'(d))^{n-1}\tilde\rho''(d)
	\end{equation*}
in $\Omega\setminus\Omega_{r_0}$, where $k_i(x_b)$, $i=1,\cdots,n-1$, are the principal curvatures of $\partial\Omega$ at $x_b$.	

By differentiation
	\begin{eqnarray*}
		&& \tilde\rho' = \varepsilon(-\rho)^{\varepsilon-1}\rho', \\
		&& \tilde\rho'' = \varepsilon(-\rho)^{\varepsilon-1}\rho'' + \varepsilon(1-\varepsilon)(-\rho)^{\varepsilon-2}(\rho')^2. 
	\end{eqnarray*}
Hence,
	\begin{equation}
	\begin{split}
		\det\,D^2v(x) &= \varepsilon^n\prod_{i=1}^{n-1}\frac{k_i(x_b)}{1-k_i(x_b)d}\left(-(-\rho)^{\varepsilon-1}\rho'\right)^{n-1}\left[(-\rho)^{\varepsilon-1}\rho'' + (1-\varepsilon)(-\rho)^{\varepsilon-2}(\rho')^2\right] \\
			&\geq \varepsilon^n\prod_{i=1}^{n-1}\frac{k_i(x_b)}{1-k_i(x_b)d} (-\rho')^{n-1}\rho''(-\rho)^{n(\varepsilon-1)} \\
			&\geq \varepsilon^nC(n,\Omega) g(d) (-\rho)^{n(\varepsilon-1)}. 
	\end{split}
	\end{equation}

By setting $\varepsilon=\frac{n}{n+1-p}$ and rescaling $v$ to $bv$ for a constant $b$ satisfying $b^{n+1-p}\varepsilon^nC(n,\Omega)\geq1$, one can see that $\det\,D^2v\geq R(x)(-1/v)^{1-p}$ in $\Omega\setminus\Omega_{r_0}$.

Observing that $v$ is a negative constant on $\partial\Omega_{r}$ for $r\in(0,r_0)$, we can extend $v$ to $\Omega_{r_0}$ so that $\det\,D^2v=\epsilon (-1/v)^{1-p}$ in $\Omega_{r_0}$. For $\epsilon$ small, $v$ is uniformly convex in $\Omega$. Similarly, by a rescaling we have $v$ is a subsolution of \eqref{PDE p<1 s} in $\Omega$ and $v=0$ on $\partial\Omega$. 

Let $w=\phi+v$, where $\phi$ is a nonpositive, convex function in $\Omega$. 
Since $w\leq v\leq0$, $(-1/v)^{1-p}\geq (-1/w)^{1-p}$, one can see that $w$ is a subsolution of \eqref{PDE p<1 s} in $\Omega$ and satisfies $w=\phi$ on $\partial\Omega$.

Last step is to use the Perron method as in \cite{CW}. 
Denote $\Phi$ by the set of all subsolutions of \eqref{PDE p<1 s} and \eqref{bc1 p<1}, and let $u(x)=\sup\{\tilde u(x) : \tilde u\in\Phi\}$.
One can easily verify that $u$ is a generalised solution of \eqref{PDE p<1 s}. Since $w\in\Phi$ we conclude that $u=\phi$ on $\partial\Omega$. 
The uniqueness is due to the comparison principle \cite{CY,GT}.
\end{proof}

One example for $g$ satisfying \eqref{cond g} is that $g(d) \leq Cd^{\beta-n-1}$ for some $\beta>0$, which is also the case considered in \cite{CY}. Notice that when $\beta\geq n+1$, $g$ is bounded in $(0,r_0]$ and thus $R$ in \eqref{PDE p<1 s} is bounded in $\Omega$.
In such a case, we can reduce the uniform convexity assumption on $\Omega$ in Lemma \ref{lexi} following the work in \cite{U1}.

\begin{lemma}
If $R$ in \eqref{PDE p<1 s} is bounded from above and $\Omega$ is a bounded convex domain, the Dirichlet problem \eqref{PDE p<1 s} and \eqref{bc1 p<1} admits a unique generalised solution. 
\end{lemma}

\begin{proof}
First, we consider the zero boundary condition $\phi\equiv0$. 
Let $\{\Omega_k\}$ be an increasing sequence of smooth uniformly convex subdomains of $\Omega$ with $\Omega=\cup\Omega_k$.
Let $\{v_k\}$ be the sequence of convex solution of \eqref{PDE p<1 s} in $\Omega_k$, and $v_k=0$ on $\partial\Omega_k$. 
Since $\{\Omega_k\}$ is an increasing sequence, $\{v_k\}$ is a decreasing sequence, by the comparison principle \cite{CY,GT}. 
We will show that $v^*=\lim_{k\to\infty}v_k$ exists, and $v^*=0$ on $\partial\Omega$.

Under a suitable coordinate we may assume that $0\in\partial\Omega$ and $\Omega\subset\{x_n>0\}$. 
Since $\Omega$ is bounded, there exists a large $K>0$ such that $\Omega\subset B_K^+(0)=B_K(0)\cap\{x_n>0\}$.
Define the barrier function
	\begin{equation*}
		w(x) = (|x'|^2-A) x_n^\delta,
	\end{equation*}
where $x'=(x_1,\cdots,x_{n-1})$ and $A > R^2$, $\delta\in(0,1)$ are to be fixed.
One can see that $w$ is convex, and by computation \cite{U1}
	\begin{equation*}
	\begin{split}
		\det\,D^2w =& \left\{2^{n-1}\delta(1-\delta)(A-|x'|^2) - 2^n\delta^2|x'|^2\right\} \\
			& \times (A-|x'|^2)^{\frac{2}{\delta}-n} \left(\frac{-1}{w}\right)^{\frac{2}{\delta}-n}.
	\end{split}
	\end{equation*}

If $n\geq2$ we choose $\delta=\frac{2}{n+1-p} \in (0,1)$, and then fix $A>K^2$ sufficiently large, so that
	\begin{equation*}
		\det\,D^2w \geq R(x)\left(\frac{-1}{w}\right)^{1-p} \quad \mbox{ in } B_K^+(0).
	\end{equation*}
When $n=1$, if $p<0$ we obtain a similar inequality with $\delta=\frac{2}{2-p}<1$, and if $0\leq p<1$ we can choose any $\delta\in(0,1)$.
Since $\Omega_k\subset B_K^+$ and $v_k=0$ on $\partial\Omega_k$, by the comparison principle we have $w\leq v_k$ in $\Omega_k$ for each $k$. 
By a similar argument at any boundary point of $\Omega$ we see that $v^*=\lim_{k\to\infty}v_k$ is well defined and is a convex generalised solution of \eqref{PDE p<1 s} satisfying $v^*=0$ on $\partial\Omega$.

For general boundary value \eqref{bc1 p<1}, let $w^*=v^*+\phi$, where $\phi\leq0$ is convex in $\Omega$. Then $w^*$ is a subsolution of \eqref{PDE p<1 s} and $w=\phi$ on $\partial\Omega$. 
Using the Perron method as in Lemma \ref{lexi}, we then obtain the generalised solution of \eqref{PDE p<1 s} and \eqref{bc1 p<1}. 
The uniqueness is due to the comparison principle \cite{CY,GT}.
\end{proof}

%%%%%%%%%%%%%%%% Sect 3.2 %%%%%%%%%%%%%%%%

\subsection{Completeness}
Since the spherical image of $M$ is strictly contained in $\mathbb{S}^n_-$, in order to be complete, $M$ must be an entire graph, namely the scaled support function $u$ must satisfy $|Du(x)|\to\infty$ as $x\to\partial\Omega$.

\begin{lemma}
Let $\Omega$ be a bounded, uniformly convex $C^2$-domain. 
Suppose that there exists a positive function $h$ in $(0,r_0]$ satisfying 
	\begin{equation}\label{cond ha}
		\int_0^{r_0}h(t)dt=\infty,
	\end{equation}
such that 
	\begin{equation}\label{cond h}
		\frac{h(d)}{R(x_b,d)}=o(1) \quad \mbox{as } d\to 0.
	\end{equation}
Then the solution $u$ of \eqref{PDE p<1 s} and \eqref{bc1 p<1}, produced by Lemma \ref{lexi}, satisfies $|Du(x)|\to\infty$ as $x\to\partial\Omega$. In particular, if $\phi=0$ on $\partial\Omega$, condition \eqref{cond h} can be reduced to
	\begin{equation}\label{cond h c}
		h(d)\leq C R(x_b,d)
	\end{equation}
for a constant $C>0$.
\end{lemma}

\begin{proof}
Adopting the notations from \S3.1.
In $\Omega\setminus\Omega_{r_0}$ we define
	\begin{equation}\label{cst sup}
		w(x) = -a\rho(d)+\phi = -a\int_0^d\left(\int_s^{r_0}h(t)dt\right)^{1/n}ds + \phi,
	\end{equation}
which is uniformly convex for $a>0$, $w=\phi$ on $\partial\Omega$ and $|Dw(x)|\to\infty$ as $x\to\partial\Omega$. 
Since $\partial\Omega\in C^2$, for sufficiently large $a>0$ we have an estimate
	\begin{equation}
	\begin{split}
		\det\,D^2w &\leq (2a)^n\prod_{i=1}^{n-1}\frac{k_i(x_b)}{1-k_i(x_b)d}(- \rho'(d))^{n-1} \rho''(d) \\
			&\leq (2a)^n C h(d),
	\end{split}
	\end{equation}
where $C$ is a constant depending on $n$ and $\partial\Omega$.
	
Recall that $\phi\leq0$. Let $\Omega':=\{x\in\Omega : w(x)<\inf_{\partial\Omega}\phi - \varepsilon\}$ be a sub-level set of $w$, which is uniformly convex. Choose $\varepsilon>0$ sufficiently small such that $\Omega_{r_0}\Subset\Omega'$ and $\Omega'\Subset\Omega$.
Note that $w$ is constant on $\partial\Omega'$, we can extend $w$ inside $\Omega'$ similarly as before and then modify $w$ to get a uniformly convex function $\tilde w\in C^2(\Omega)$ such that $\tilde w=w$ in $\Omega\setminus\Omega'$.
Near $\partial\Omega$, observe that in $\Omega\setminus\Omega'$
	\begin{equation*}
		R(x_b,d)\left(-\frac{1}{w}\right)^{1-p} \geq R(x_b,d)C_1,
	\end{equation*}
where $C_1=(-\inf_{\partial\Omega}\phi +\varepsilon)^{p-1}>0$ is a finite constant. 
Therefore, $\tilde w$ is a supersolution of \eqref{PDE p<1 s}--\eqref{bc1 p<1}, and by the comparison principle, $\tilde w\geq u$ in $\Omega$.
Hence, $|Du(x)|\to\infty$ as $x\to\partial\Omega$. 

In the special case $\phi\equiv0$, we set $a=1$ in \eqref{cst sup}, and at the end, replace $\tilde w$ by $b\tilde w$ with $b>0$ sufficiently small, so that we can obtain a supersolution.	
\end{proof}

An example for $h$ satisfying \eqref{cond ha} is that $h(t)=t^{-\alpha}$ for some $\alpha>1$.
%(one may need $\alpha\leq n+1$ such that $|w|<\infty$ is well defined.)
Alternatively, when studying the homothetic solutions to Gauss curvature flow, Urbas \cite{U1} proved that if $\hat f\equiv1$ in $\Omega$, $\phi=0$ on $\partial\Omega$, then for a range of $p$, $|Du(x)|\to\infty$ as $x\to\partial\Omega$. 
Inspired by that, we have the following results. 

\begin{lemma}\label{lUr1}
Let $\Omega$ be a bounded convex domain and $\partial\Omega\in C^{1,1}$.
Assume that $\phi=0$ on $\partial\Omega$, $\hat f>0$ in $\overline\Omega$.
Then, when $p\leq0$, the solution $u$ of \eqref{PDE p<1 s}--\eqref{bc1 p<1} satisfies $|Du(x)|\to\infty$ as $x\to\partial\Omega$.
\end{lemma}

\begin{proof}
The proof follows from \cite{U1}.
Let $x_0\in\partial\Omega$ and let $B$ be an interior ball at $x_0$, i.e., $B\subset\Omega$ and $\partial B\cap\partial\Omega=\{x_0\}$.
From assumptions, $\underline{R}=\inf_{x\in B}R(x)$ is a positive constant, where $R$ is defined in \eqref{def r}.
Let $w$ be the unique convex solution of the Dirichlet problem
	\begin{eqnarray}
		&& \det\,D^2w = \underline{R}\left(\frac{-1}{w}\right)^{1-p} \quad\mbox{ in }B, \label{supw}\\
		&& w=0	\qquad \mbox{ on } \partial B,	\nonumber
	\end{eqnarray}
The solvability is due to Lemma \ref{lexi}, and the solution $w$ is radially symmetric since the above problem has at most one convex solution.
Hence, $w$ is a supersolution of \eqref{PDE p<1 s} and \eqref{bc1 p<1}, and $w\geq u$ in $B$ by the comparison principle.
So, it suffices to show that $|Dw(x)|\to\infty$ as $x\to x_0$. 
Suppose on the contrary that $N=\sup_B|Dw|<\infty$, then $|w|\leq Nd$ where $d=\dist(x,\partial B)$, and
	\begin{equation}
	\begin{split}
		\omega_nN^n &= |Dw(B)| \\
			&= \int_B \det\,D^2w \\
			&= \underline{R} \int_B \left(\frac{-1}{w}\right)^{1-p} \\
			&\geq \underline{R}N^{p-1}\int_B d^{p-1}.
	\end{split}
	\end{equation}
When $p\leq0$, the last integral is infinite, which gives a contradiction. 
Therefore, $N=\infty$, and Lemma \ref{lUr1} is proved.
\end{proof}

The following lemma shows that in order for $M$ to be complete, it is necessary to have $\hat f(\xi)\to\infty$ as $\xi\to\partial D$, where $D\Subset\mathbb{S}^n_-$ is the spherical image of $M$.

\begin{lemma}\label{lUr2}
If $0<p<1$, $\phi=0$ on $\partial\Omega$, $\hat f$ is bounded above in $\Omega$, and $\Omega$ satisfies a uniform enclosing sphere condition (namely, there exists a $K>0$ such that for each $x_0\in\partial\Omega$ there is a ball $B$ of radius $K$ with $\Omega\subset B$ and $\partial B\cap\partial\Omega=\{x_0\}$), 
then the solution $u$ of \eqref{PDE p<1 s}--\eqref{bc1 p<1} satisfies $\sup_\Omega|Du|\leq C$.
\end{lemma}

\begin{proof}
Let $B$ be an enclosing ball at any point $x_0\in\partial\Omega$.
From assumptions, $\overline{R}=\sup_{x\in\Omega}R(x)$ is a positive constant. 
Replacing $\underline{R}$ in \eqref{supw} by $\overline{R}$, the convex solution $w$ will be a subsolution of \eqref{PDE p<1 s} and \eqref{bc1 p<1}, and a gradient bound for $u$ follows if we can prove $N=\sup_\Omega|Dw|<\infty$. 
Since $w$ is convex, $w\not\equiv0$ and $w=0$ on $\partial B$, we have $|w|\geq\theta d$ for some positive constant $\theta$, where $d=\dist(x,\partial B)$.
Proceeding as above, we now obtain
	\begin{equation*}
		\omega_nN^n \leq \overline{R}\theta^{p-1} \int_B d^{p-1}.
	\end{equation*}
The last integral is finite if $p>0$. Therefore, we have a gradient bound for $w$, and hence also for the solution $u$.
This completes the proof of Lemma \ref{lUr2}.
\end{proof}

%%%%%%%%%%%%%%%%%%%%%%%%%%%%%%%%%%%%%%%%%%%%%%%%%%%%%%
%%%%%%%%%%%%%%%%%%%%%%%%%%%%%%%%%%%%%%%%%%%%%%%%%%%%%%
%%%%%%%%%%%%%%%%%%%%%%%%%%%%%%%%%%%%%%%%%%%%%%%%%%%%%%
\vspace{10pt}

\section{The case of $p>1$}

Recall that in \S1 we know that for a complete, noncompact, convex hypersurface $M$ in $\mathbb{R}^{n+1}$, by a suitable rotation of coordinates its spherical image $D\subset\mathbb{S}^n_-$ satisfies one and exactly one of three cases (I), (II) and (III).
Given a function $f$ on $D$, we investigate the existence of $M$ such that $f=H^{1-p}/K$ is the $p$-curvature function of $M$, where $H$ is the support function and $K$ is the Gauss curvature of $M$.
When $M$ is $C^2$ smooth, a function $f$ is the $p$-curvature function of $M$ if it satisfies equation \eqref{PDEH}, or \eqref{pde2} under an orthonormal frame field on $\mathbb{S}^n$. 
By the homogeneous extension \eqref{Htou}, one has $u$ satisfies equation \eqref{PDEu} in the domain $\Omega$, and $u^*$ satisfies \eqref{PDEu*} in $\Omega^*$. 
The hypersurface $M$ is then the graph of $u^*$ over $\Omega^*$.

Notice that by the convexity of $M$, $K$ is always nonnegative. However, the sign of the support function $H$ depending on the relative position of $M$ and the origin. As seen in \S3, the above problem is equivalent to \eqref{p<1} that
	\begin{equation}\label{p>1}
		\frac{1}{H} = \tilde f K^{\frac{1}{p-1}},
	\end{equation}
where $\tilde f$ is the given function on $D$. 
Similar to the nonexistence result in the case of $p<1$, i.e. Lemma \ref{nonexistence}, when $p>1$ we have the following analogous result. 

\begin{lemma}[Nonexistence] \label{non2}
If $\tilde f\leq0$ is a nonpositive function on $D$, satisfying $\tilde f\in L^{p-1}(D)$, there does not exist any complete, noncompact, strictly convex hypersurface $M\in C^2$ satisfying \eqref{p>1}.
\end{lemma}

\begin{proof}
%The proof is based on \cite{U2}. 
Since $M$ is in the class $C^2$, by \eqref{p>1} we have $H<0$ and $0\notin M$. 
Hence, every tangent hyperplane $T$ of $M$ must pass between $0$ and $M$, so $\dist(0,T)\leq \dist(0,M)=:d$. 
Thus, $-H^{-1}\geq d^{-1}$. 
Since $M$ has infinite $n$-dimensional Hausdorff measure $\mathcal{H}^n$, by integrating we have
	\begin{equation*}
	\begin{split}
		\int_D |\tilde f|^{p-1} dx &= \int_M (-\tilde f)^{p-1} K  d\mathcal{H}^n\\
			&= \int_M (-\frac{1}{H})^{p-1} d\mathcal{H}^n \\
			&\geq \int_M (\frac{1}{d})^{p-1} d\mathcal{H}^n= \infty,
	\end{split}
	\end{equation*}
where $dx$ is the spherical measure of $\mathbb{S}^n$. The above inequality contradicts the assumption and thus completes the proof of Lemma \ref{non2}.
\end{proof}

In the subsequent context, we assume $\tilde f\geq0$. Let $f=\tilde f^{p-1}$, we consider the equation
	\begin{equation}\label{p>1 H}
		\frac{1}{K} = f H^{p-1},
	\end{equation}
where $p>1$. This is a counterpart of Equation \eqref{p<1a} in the case of $p<1$.
By the rescaling \eqref{Htou}, Equation \eqref{p>1 H} is equivalent to \eqref{PDEu}, namely
	\begin{equation}\label{PDEu p>1}
		\det\,D^2u = (1+|x|^2)^{-\frac{n+p+1}{2}} f\left(\frac{x,-1}{\sqrt{1+|x|^2}}\right)\,u^{p-1},\quad x\in\Omega,
	\end{equation}
where $\Omega = \{\lambda\xi : \xi\in D, \lambda>0\} \cap \{x_{n+1}=-1\}$, $u(x) = H(x,-1)$, and $x=(x_1,\cdots,x_n)$.
Depending on the spherical image $D\subset\mathbb{S}^n_-$ of $M$, let's consider the cases (I), (II) and (III) separately in the following. 
	
\subsection{Type I}	
For type I hypersurfaces, $D$ is strictly contained in $\mathbb{S}^n_-$, $\Omega$ is a bounded convex domain in $\mathbb{R}^n$. 
%However, we need some appropriate boundary condition to guarantee $M$ is complete.
%For type I hypersurfaces, 
It is clear that $M$ is complete if and only if $\Omega^*=\mathbb{R}^n$, where $\Omega^*=Du(\Omega)$.
Thus we pose the following boundary conditions associated with equation \eqref{PDEu p>1}
	\begin{eqnarray}
		&& |Du(x)| \to \infty \quad \mbox{as } x\to\partial\Omega, \label{by complete} \\
		&& u(x) = \phi(x) \quad x\in\partial\Omega, \label{by Diri}
	\end{eqnarray}
where $\phi$ is assumed to be a positive, convex function in $\overline\Omega$.  
We remark that in \cite{U2}, when studying homothetic solutions of negative powered Gauss curvature flows, Urbas considered the above boundary value problem with $f\equiv1$ and $\phi=\infty$ on $\partial\Omega$.

The solvability of Dirichlet problem \eqref{PDEu p>1} and \eqref{by Diri} has been previously obtained in \cite[\S7]{CNS1}, \cite{Guan98} and \cite{Lio} under appropriate assumptions, especially $f$ is required to be positive and bounded. 
However, if the solution $u$ satisfies \eqref{by complete}, by integrating equation \eqref{PDEu p>1},
	\begin{equation*}
	\begin{split}
		\infty &= |Du(\Omega)| \\
			&= \int_\Omega \det\,D^2u \leq C\int_\Omega f u^{p-1}.
	\end{split}
	\end{equation*}
Notice that $0<u\leq\sup_{\partial\Omega}\phi$. 
If $\sup_{\partial\Omega}\phi<\infty$, then it is necessary to have $f(x)\to\infty$ as $x\to\partial\Omega$.
In the following, due to some technical differences, we consider two cases $p\in(1,n+1)$ and $p\in (n+1,\infty)$ separately. 
In each case, we first show the solvability of Dirichlet problem \eqref{PDEu p>1} and \eqref{by Diri}, and then prove that such an obtained solution $u$ satisfies \eqref{by complete}. 
Hence, Theorem \ref{thm I} is proved.
Our approach is inspired by the work of Chou-Wang \cite{CW95}, in which they considered the special case $p=1$, see also \cite{P2} and references therein.

\subsubsection{The case of $1<p<n+1$}

Write 
	\begin{equation*}
		R(x) = (1+|x|^2)^{-\frac{n+p+1}{2}} f\left(\frac{x,-1}{\sqrt{1+|x|^2}}\right),\quad x\in\Omega.
	\end{equation*}
%and assume the boundary value $\phi$ belongs to $C^\infty(\overline\Omega)$ and is convex. 
Equation \eqref{PDEu p>1} can be reduced to
	\begin{equation}\label{PDEu1}
		\det\,D^2u = u^{p-1}R(x),\quad x\in\Omega. 
	\end{equation}
As in \S3, let $\Omega_r=\{x\in\Omega : \dist(x,\partial\Omega)>r\}$. 
%For $r_0>0$ small depending on the geometry of $\Omega$, $\Omega_r$ is still uniformly convex for $r\in(0,2r_0)$. 
For $x\in\Omega\setminus\Omega_{r_0}$, $r_0>0$ small, $x$ can be represented uniquely by $x_b+d n(x_b)$, where $x_b\in\partial\Omega$, $d=\dist(x,\partial\Omega)$, and $n(x_b)$ is the unit inner normal at $x_b$.
For a function $f$ defined near $\partial\Omega$ we write $f(x)=f(x_b,d)$.

\begin{lemma}\label{ex p<n+1}
Assume that $1\leq p <n+1$ and $\Omega\in C^2$ is uniformly convex. 
Suppose there exists a positive function $g$ in $(0,r_0]$, satisfying
	\begin{equation}\label{cong}
		\int_0^{r_0}\left(\int_s^{r_0}g(t)dt\right)^{1/n}ds < \infty,
	\end{equation} 
such that
	\begin{equation*}
		R(x_b,d) \leq g(d).
	\end{equation*}	
Then \eqref{PDEu1} admits a unique generalised solution $u$ in $C(\Omega)$ and $u=\phi$ on $\partial\Omega$.
\end{lemma}

\begin{proof}
Similarly as in \cite{CW95}, for $x=x_b+dn(x_b)$ in $\Omega\setminus\Omega_{r_0}$ we define
	\begin{equation}\label{funv}
		v(x)=\rho(d)= - \int_0^d\left(\int_s^{r_0}g(t)dt\right)^{1/n}ds,
	\end{equation}
and have
	\begin{equation}
		\det\,D^2v(x) = \prod_{i=1}^{n-1}\frac{k_i(x_b)}{1-k_i(x_b)d}(-\rho'(d))^{n-1}\rho''(d)
	\end{equation}
in $\Omega\setminus\Omega_{r_0}$, where $k_i(x_b)$, $i=1,\cdots,n-1$, are the principal curvatures of $\partial\Omega$ at $x_b$.

Next, we extend $v$ inside $\Omega_{r_0}$. Note that $v=-G_0$ is a constant on $\partial\Omega_{r_0/2}$. We extend $v$ to $\Omega_{r_0/2}$ so that $\det\,D^2v=\varepsilon>0$ in $\Omega_{r_0/2}$. 
For $\varepsilon$ small, $v$ is uniformly convex in $\Omega$. 
By the uniform estimate, we have
	\begin{equation}
		\sup |v| \leq G_0 + C|\Omega|^{2/n},
	\end{equation}
for some constant $C$ depending only on $n, \varepsilon$.

Let 
	\begin{equation*}
		w(x) = \phi(x) + A v(x), \quad x\in\Omega.
	\end{equation*}
Then $w$ is convex in $\Omega$, $w=\phi$ on $\partial\Omega$, and 
	\begin{equation}
	\begin{split}
		\sup|w| &\leq |\phi|_0 + A|v| \\
			&\leq |\phi|_0 + A(G_0 +C|\Omega|^{2/n}).
	\end{split}
	\end{equation}

By computation we have the left hand side of equation \eqref{PDEu1}
	\begin{equation}\label{detu}
		\det\,D^2w \geq A^n\det\,D^2v \geq \left\{\begin{array}{ll}
		A^n\varepsilon & \mbox{in }\Omega_{r_0} \\
		A^nC_1g &\mbox{in }\Omega\setminus\Omega_{r_0},
		\end{array}
		\right.
	\end{equation}	
where $C_1$ is a constant depending on $n, r_0$ and $\partial\Omega$.
Meanwhile, the right hand side 
	\begin{equation}
		w^{p-1}R(x) \leq \sup|w|^{p-1} R(x) \leq \left\{\begin{array}{ll}
		\sup|w|^{p-1}\bar R & \mbox{for }x\in\Omega_{r_0} \\
		\sup|w|^{p-1} R(x) &\mbox{for }x\in\Omega\setminus\Omega_{r_0},
		\end{array}
		\right.
	\end{equation}
where $\bar R=\sup_{x\in\Omega_{r_0}}R(x)$ is finite.
	
Since $p<n+1$, we can choose $A$ sufficiently large such that $\det\,D^2w\geq w^{p-1}R(x)$ in $\Omega$. 
This means that $w$ is a subsolution of \eqref{PDEu1} and $w=\phi$ on $\partial\Omega$. 
Last step is to use the Perron method, which requires $p\geq1$. 
Denote $\Phi$ by the set of all subsolutions of \eqref{PDEu1} and \eqref{by Diri}, and let $u(x)=\sup\{\tilde u(x) : \tilde u\in\Phi\}$.
One can easily verify that $u$ is a generalised solution of \eqref{PDEu1}. Since $w\in\Phi$ we conclude $u=\phi$ on $\partial\Omega$. 	
\end{proof}

\begin{lemma}
Suppose that there exists a positive function $h$ in $(0,r_0]$ satisfying 
	\begin{equation}
		\int_0^{r_0}h(t)dt=\infty,
	\end{equation}
such that 
	\begin{equation*}
		R(x_b,d)\geq h(d).
	\end{equation*}
Then the solution $u$ produced by the above lemma satisfies \eqref{by complete}.
\end{lemma}

\begin{proof}
For this proof we need an upper barrier function. Recall that $\phi>0$ on $\partial\Omega$.
Introduce the function $v$ as before, where $g$ in \eqref{funv} is now replaced by $h$, namely
	\begin{equation*}
		v(x)=\rho(d)= - \int_0^d\left(\int_s^{r_0}h(t)dt\right)^{1/n}ds,\quad x\in\Omega\setminus\Omega_{r_0}.
	\end{equation*}
Then $|Dv(x)|\to\infty$ as $x\to\partial\Omega$.
Extend $v$ to $\Omega_{r_0}$ as in the previous proof and then modify $v$ to get a uniformly convex function $\tilde v\in C^2(\Omega)$ so that $\tilde v= v$ in $\Omega\setminus\Omega_{r_0/2}$. 

For any point $x_0\in\partial\Omega$ we shall assume $x_0=0$ and the positive $x_n$-axis is in the inner normal direction.
Define
	\begin{equation}
		\hat v(x):=\eta\tilde v(x)+\phi(0)+ x\cdot D\phi(0) + Kx_n,\quad x\in\Omega,
	\end{equation}
where $K>0$ is a constant. 
As $\phi\in C^2(\overline\Omega)$ and $\partial\Omega\in C^2$ is uniformly convex, we can choose $K$ large enough such that $\hat v\geq\phi$ on $\partial\Omega$ and $\hat v(x_0)=\phi(x_0)$. 
Then by choosing $\eta>0$ small enough, we also have $\hat v\geq v_0>0$ in $\Omega$.

Using similar computations as before, we have
	\begin{equation*}
		\det\,D^2\hat v \leq \eta^n\det\,D^2\tilde v \leq \left\{\begin{array}{ll}
		\eta^n\varepsilon & \mbox{in }\Omega_{r_0} \\
		\eta^nC_1h &\mbox{in }\Omega\setminus\Omega_{r_0},
		\end{array}
		\right.
	\end{equation*}
where $C_1$ is a constant depending on $n, r_0$ and $\partial\Omega$.
For the right hand side we have
	\begin{equation*}
		\hat v^{p-1}R(x) \geq v_0^{p-1}R(x) \geq \left\{\begin{array}{ll}
		v_0^{p-1}\underline{R} & \mbox{for }x\in\Omega_{r_0} \\
		v_0^{p-1} R(x) &\mbox{for }x\in\Omega\setminus\Omega_{r_0},
		\end{array}
		\right.
	\end{equation*}
where $\underline{R}=\inf_{x\in\Omega_{r_0}}R(x)$ is positive and finite. 

Therefore, by choosing $K$ sufficiently large and $\eta$ sufficiently small, using the comparison principle \cite{GT} we obtain $\hat v\geq u$ in $\Omega$.  Hence $|Du(x)|\to\infty$ as $x\to x_0$. 
\end{proof}
%By choosing a sufficiently small $\eta>0$, from \eqref{detu} we have
%	\begin{equation}\label{detc}
%		\det\,D^2(\eta \tilde v) < \det\,D^2w.
%	\end{equation}
%
%Now, we deal with the boundary value. Note that $\tilde v=0$, $w=\phi$ on $\partial\Omega$, where $\phi\in C^2(\overline\Omega)$ and $\Omega$ is uniformly convex. 

%By adding an affine function to $\eta\tilde v$ such that $\hat v(x):=\eta\tilde v(x)+\phi(0)+Kx_n$, for $K>0$ large enough we have
%	\begin{equation*}
%		\det\,D^2\hat v < \det\,D^2w \mbox{ in }\Omega;\qquad \hat v>w \mbox{ on }\partial\Omega; \quad\mbox{and}\quad \hat v(0)=w(0).
%	\end{equation*}
%By the comparison principle \cite{GT}, $\hat v\geq w$ in $\Omega$. Hence $|Dw(x)|\to\infty$ as $x\to x_0$. 

\subsubsection{The case of $p>n+1$}

To obtain existence, we adopt a different approach of constructing subsolutions. 
Let's define
	\begin{equation}
		\rho(d)= - \int_0^d\left(\int_s^{r_0}g(t)dt\right)^{1/n}ds.
	\end{equation}
Assume $\phi=\phi_0>0$ is a constant on $\partial\Omega$. Define
	\begin{equation}\label{defv}
		v(x) = (-A\rho(d)+\phi_0^{-\frac{1}{\delta}})^{-\delta},\quad\mbox{ in }\Omega\setminus\Omega_{r_0},
	\end{equation}
where $\delta>0$, $A>0$ are constants to be determined.

Setting $\delta=\frac{n}{p-n-1}$, by computation we have
	\begin{equation}
	\begin{split}
		\det\,D^2 v &\geq  A^n\delta^n v^{\frac{n(\delta+1)}{\delta}}\prod_{i=1}^{n-1}\frac{\kappa_i}{1-\kappa_id}(-\rho')^{n-1}\rho'' \\
			&\geq A^n\delta^n C(n,\Omega) g(d) v^{p-1},	
	\end{split}
	\end{equation}
where $C$ is a constant depending only on $n$ and $\Omega$.
Choosing $A$ sufficiently large and by extending $v$ inside $\Omega_{r_0}$ as before, we then obtain a subsolution. 
Note that $0<v\leq \phi_0$ in $\Omega$ and $v=\phi_0$ on $\partial\Omega$. 
The existence of solution $u$ thus follows by the Perron process. 

For a general $\phi>0$ on $\partial\Omega$, we need modify $v$ in \eqref{defv}.
For a point $x_0\in\partial\Omega$, we may assume $x_0=0$ and the positive $x_n$-axis is in the inner normal direction. 
Let $\phi_0=\phi(0)$.
Define
	\begin{equation}\label{defv1}
		v(x) = (-A\rho(d)+\phi_0^{-\frac{1}{\delta}}+Kx_n-\frac{1}{\delta}\phi_0^{-\frac{1}{\delta}-1}x\cdot D\phi(0))^{-\delta},\quad\mbox{ in }\Omega\setminus\Omega_{r_0},
	\end{equation}
where $K>0$ is chosen sufficiently large such that $v\leq\phi$ on $\partial\Omega$ and $v=\phi$ at $x_0$. 
By choosing $\delta=\frac{n}{p-n-1}$ and $A$ sufficiently large as above, we have $v$ is a solution. Therefore, the existence of solution $u$ follows. 

For completeness, 
Lemma 5.2 applies in this case, so we have $|Du(x)|\to\infty$ as $x\to\partial\Omega$, and obtain the completeness.

%%%%%%%%%%%%%%%%%%%%%%%%%%%%%%%%%%%%%

\subsection{Type II}

Next, we consider type II hypersurfaces. In this case, we investigate the entire solution of \eqref{PDEu p>1}, i.e.
	\begin{equation}\label{PDE II}
		\det\,D^2u = (1+|x|^2)^{-\frac{n+p+1}{2}} f\left(\frac{x,-1}{\sqrt{1+|x|^2}}\right)\,u^{p-1} \quad \mbox{ in }\Omega=\mathbb{R}^n.
	\end{equation}
When $p>n+1$, we prove the existence of a solution by constructing suitable upper and lower barriers.

Assume that $f$ satisfies the asymptotic growth condition
	\begin{equation}\label{f asym}
		f(x) \sim (1+|x|^2)^q \quad \mbox{ as } x\to\infty,
	\end{equation}
where $q\in(0,1)$ is a constant. 
Note that this is equivalent to $f(X)\sim |X_{n+1}|^{-2q}$.
We remark that it is necessary to have a growth condition on $f$. 
Otherwise, by integration one can see that when $\Omega^*$ is bounded, $f$ is bounded, there doesn't exist a complete noncompact hypersurface $M$ satisfying \eqref{PDE II}, (see \cite{U2} for the case of $f\equiv1$).
To prove this claim, it is convenient to use Equation \eqref{PDEu*} for the dual function $u^*$. In that case, $M$ is a graph of $u^*$ over $\Omega^*$.
If $\Omega^*$ is bounded, let $y_0\in\partial\Omega^*$. Since $u^*$ is convex, there exists a constant $C_0>0$ such that
	\begin{equation*}
		u^* \geq -C_0 \quad\mbox{ on }\Omega^*\cap B_1(y_0).
	\end{equation*}
Let $P\in \tilde M := M\cap\left(B_1(y_0)\times\mathbb{R}\right)$.
Assuming $M$ is a complete, noncompact hypersurface, we compute its support function $H$ at $P$ and have
	\begin{equation*}
		H|_P = \frac{y\cdot Du^* - u^*}{\sqrt{1+|Du^*|^2}} \leq |y|+C_0 \leq 1+|y_0|+C_0.
	\end{equation*}
Consequently, $H^{-1} \geq c_0>0$ in a neighbourhood $G\subset\tilde M$. 
Similarly to the nonexistence Lemmas \ref{nonexistence} and \ref{non2}, by integrating \eqref{p>1} we obtain
	\begin{equation*}
		c_0^{p-1}\mathcal{H}^n(G) \leq \int_D fK d\mu= \int_D f dx,
	\end{equation*}
where $D\subset\mathbb{S}^n_-$ is the spherical image of $G$, $d\mu = K^{-1}dx$ is the area measure, and $dx$ is the spherical measure.
As $\mathcal{H}^n(G)=\infty$, so the function $f$ cannot be bounded. 

From \eqref{PDE II} and \eqref{f asym}, an upper (or lower) barrier is a function satisfying 
	\begin{equation*}
		\det\,D^2u \leq (1+|x|^2)^{-\gamma}u^{p-1} \qquad (\mbox{or } \geq)
	\end{equation*}
as $x\to\infty$, where $\gamma:=\frac{n+p+1}{2}-q$. 
In a bounded domain, one can always construct such a barrier by rescaling $u$ to $\lambda u$ for a suitable constant $\lambda$, provided $p\neq n+1$. 
	
Now, let's consider the function 
	\begin{equation*}
		w(x) = (1+|x|^2)^\delta
	\end{equation*}
where $\delta>1/2$ is to be chosen. Clearly $w$ is a convex function. 

By computations
	\begin{equation*}
		D_{ij}w = 2\delta(1+|x|^2)^{\delta-1}\delta_{ij} + 4\delta(\delta-1)(1+|x|^2)^{\delta-2}x_ix_j,
	\end{equation*}
where $\delta_{ij}$ is the Kronecker delta. Hence,
	\begin{equation*}
	\begin{split}
		\det\,D^2w &= (2\delta)^n(1+|x|^2)^{n(\delta-1)}\left(\frac{1+(2\delta-1)|x|^2}{1+|x|^2}\right) \\
				&= C(n,\delta) w^{\frac{n(\delta-1)+\gamma}{\delta}} (1+|x|^2)^{-\gamma},
	\end{split}
	\end{equation*}
where $C(n,\delta)$ is a positive constant bounded by $C_1 \leq C(n,\delta) \leq C_2$, and
	\begin{eqnarray*}
		&& C_1:=(2\delta)^n\inf_{x\in\mathbb{R}^n}\left\{\frac{1+(2\delta-1)|x|^2}{1+|x|^2}\right\}, \\
		&& C_2:=(2\delta)^n\sup_{x\in\mathbb{R}^n}\left\{\frac{1+(2\delta-1)|x|^2}{1+|x|^2}\right\}.
	\end{eqnarray*}
Choose $\delta=(\gamma-n)/(p-n-1)$ such that ${\frac{n(\delta-1)+\gamma}{\delta}}=p-1$. One can see that as far as $q<1$,
	\begin{equation*}
		\delta > \frac{\frac{n+p+1}{2}-n-1}{p-n-1} = \frac{1}{2}.
	\end{equation*}
By a rescaling we obtain that $\overline{w}=\overline\mu w$ is a convex supersolution of \eqref{PDE II} for $\overline\mu^{p-n-1}\geq C_1$, while $\underline{w}=\underline\mu w$ is a convex subsolution of \eqref{PDE II} for $0<\underline\mu^{p-n-1}\leq C_2$.
Since $\underline\mu\leq\overline\mu$, we have $\underline w\leq\overline w$.
Let $\phi$ be any smooth function such that $\underline w\leq\phi\leq\overline w$ in $\mathbb{R}^n$. 
By \cite{CNS1} the Dirichlet problem
	\begin{eqnarray*}
		\det\,D^2w_k \!\!&=&\!\! (1+|x|^2)^{-\frac{n+p+1}{2}} f\,w_k^{p-1} \quad\mbox{in }B_{2^k}(0), \\
			w_k	\!\!&=&\!\! \phi \quad\mbox{on }\partial B_{2^k}(0),
	\end{eqnarray*}
has a unique convex solution $w_k\in C^\infty(B_{2^k})$, $\underline w\leq w_k\leq\overline w$ in $B_{2^k}$. 
From this there exists a subsequence of $\{w_k\}$ converges locally in any $C^l$ form to a convex solution $u\in C^\infty(\mathbb{R}^n)$ of \eqref{PDE II}, and $\underline w\leq u \leq\overline w$ in $\mathbb{R}^n$.
Thus $u(x)/\sqrt{1+|x|^2}\to\infty$ as $|x|\to\infty$, $\Omega^*=Du(\Omega)=\mathbb{R}^n$, and hence the corresponding hypersurface $M$ is complete.

%%%%%%%%%%%%%%%%%%%%%%%%%%%%
In fact, any admissible solution $u$ of \eqref{PDE II} with $Du(\Omega)=\mathbb{R}^n$ must satisfies
	\begin{equation}\label{u asym}
		u(x)/\sqrt{1+|x|^2}\to\infty \quad\mbox{ as }|x|\to\infty.
	\end{equation}
Otherwise, if this is not true, there exists a sequence $\{z_k\}\subset\Omega$ such that $|z_k|\to\infty$ and for each $k$, $u(z_k)\leq C|z_k|$ for some constant $C$. 
By choosing a subsequence and making a rotation of coordinates if necessary we may assume that $z_k/|z_k|\to e_n=(0,\cdots,0,1)$. 
Let
	\begin{equation*}
		x_{n+1} = a_0 + \langle a, x \rangle = a_0 + \sum_{i=1}^n a_ix_i
	\end{equation*}
be the graph of any tangent hyperplane to graph $u$. Then
	\begin{equation*}
		a_0 + \langle a, z_k \rangle \leq C|z_k| 
	\end{equation*}
for each $k$, so dividing by $|z_k|$ and letting $k\to\infty$ we obtain $a_n\leq C$. 
This implies that $Du(\Omega)\cap\{x_n>C\}=\emptyset$, which contradicts with $Du(\Omega)=\mathbb{R}^n$.
This proves \eqref{u asym}. 

Therefore, we have the following existence result, which is equivalent to Theorem \ref{thm II}.
\begin{theorem}
When $p>n+1$, $f$ satisfies the asymptotic growth condition \eqref{f asym}, there exists a complete noncompact hypersurface $M$ whose support function is a solution of \eqref{PDE II}.
\end{theorem}

%%%%%%%%%%%%%%%%%%%%%%%%%%

\subsection{Type III}

This case can be handled similarly as in \cite{CW95}.
We observe that for a type III hypersurface, $\Omega$ is of the form $\omega\times\mathbb{R}^m$ for some $m<n$, where $\omega$ is a bounded convex domain in $\mathbb{R}^{n-m}$. 
Near $\partial\omega$ we may write $\tilde x=(x_1,\cdots,x_{n-m})=\tilde x_b+dn(\tilde x_b)$ analogously as before. 
Correspondingly the boundary conditions \eqref{by complete} and \eqref{by Diri} are imposed on $\partial\omega$
	\begin{eqnarray}
		&& |Du(x)| \to \infty \quad \mbox{as } x\to\partial\omega, \label{by complete 3} \\
		&& u(x) = \phi(x) \quad x\in\partial\omega, \label{by Diri 3}
	\end{eqnarray}
where $\phi$ is prescribed on $\partial\omega$. 
Then by following the lines in \S4.1, we have

\begin{theorem}\label{thm III}
Let $p\geq1$, $\neq n+1$, $\Omega=\omega\times\mathbb{R}^m$, where $\omega$ is a uniformly convex $C^2$-domain in $\mathbb{R}^{n-m}$.
Suppose that $\phi$ can be extended to $\Omega$ so that $D^2\phi(x)\geq\delta_0I$ for some positive constant $\delta_0$, where $I$ is the identity matrix. 
Suppose moreover there exist two positive functions $g$ and $h$ defined in $(0,r_0]$, $r_0>0$, satisfying 
	\begin{eqnarray}
		&& \int_0^{r_0}\left(\int_s^{r_0}g(t)dt\right)^{1/(n-m)}ds < \infty, \\
		&& \int_0^{r_0}h(t)dt=\infty,  
	\end{eqnarray}
such that
	\begin{equation}
		h(d) \leq R(\tilde x_b, d) \leq g(d),\quad\mbox{ where } \tilde x=\tilde x_b+dn(\tilde x_b),
	\end{equation}
near $\partial\Omega$. 
Then there exists a unique solution $u$ of \eqref{PDEu p>1}, \eqref{by complete 3} and \eqref{by Diri 3} in $C(\overline\Omega)\cap C^{2,\alpha}(\Omega)$.
\end{theorem}

%\newpage
\bibliographystyle{amsplain}

\end{document}